\documentclass[reqno]{amsart}

\usepackage{amssymb}
\usepackage{graphicx}
\usepackage{amscd}
\usepackage[pagebackref]{hyperref}
\usepackage{color}
\usepackage{tabularx}
\usepackage[table]{xcolor}
\usepackage{float}
\usepackage{graphics,amsmath,amssymb}
\usepackage{amsthm}
\usepackage{amsfonts}
\usepackage{latexsym}
\usepackage{epsf}
\usepackage{xifthen}
\usepackage{mathrsfs}
\usepackage{dsfont}
\usepackage{makecell}
\usepackage{subfig}
\usepackage{amsmath}
\usepackage{listings}
\usepackage{etoolbox}
\usepackage{fancyhdr}
\usepackage{pdflscape}
\usepackage[title,toc,titletoc]{appendix}
\usepackage{enumitem}
\usepackage[noadjust]{cite}
\usepackage{tikz}
\usetikzlibrary{automata,positioning,arrows}
\usepackage{young}
\usepackage[object=vectorian]{pgfornament} 
\usepackage{lipsum,tikz}
\usepackage{multirow}
\usepackage[OT2,T1]{fontenc}
\usepackage{mathtools}
\usepackage{ytableau}

\hypersetup{
	colorlinks=true, 
	linktoc=all, 
	linkcolor=blue} 

\numberwithin{equation}{section}

\theoremstyle{theorem}
\newtheorem{theorem}{Theorem}[section]
\newtheorem*{theorem*}{Theorem}

\newtheorem{corollary}[theorem]{Corollary}

\newtheorem{innercustomgeneric}{\customgenericname}
\providecommand{\customgenericname}{}
\newcommand{\newcustomtheorem}[2]{%
	\newenvironment{#1}[1]
	{%
		\renewcommand\customgenericname{#2}%
		\renewcommand\theinnercustomgeneric{##1}%
		\innercustomgeneric
	}
	{\endinnercustomgeneric}
}
\newcustomtheorem{ctheorem}{Theorem}
\newcustomtheorem{clemma}{Lemma}

\theoremstyle{definition}

\newtheorem*{example*}{Example}
\newtheorem*{examples*}{Examples}
\newtheorem{remark}[theorem]{Remark}
\newtheorem*{remark*}{Remark}
\newtheorem*{remarks*}{Remarks}
\newtheorem*{note*}{Note}

\newtheoremstyle{named}{}{}{\itshape}{}{\bfseries}{.}{.5em}{\thmnote{#3} #1}
\theoremstyle{named}

\newtheoremstyle{customized}{}{}{\itshape}{}{\bfseries}{.}{.5em}{\thmnote{#3}}
\theoremstyle{customized}

\newcommand{\arxiv}[1]{\href{https://arxiv.org/abs/#1}{arXiv:#1}}

\DeclareMathAlphabet{\mydutchcal}{U}{dutchcal}{m}{n}

\newcommand{\qbinom}[2]{{\genfrac{[}{]}{0pt}{}{#1}{#2}}}

\newcommand{\dd}{\operatorname{d}}

\title{New central $q$-binomial identities}

\author[S. Chern]{Shane Chern}
\address[S. Chern]{Fakult\"at f\"ur Mathematik, Universit\"at Wien, Oskar-Morgenstern-Platz 1, Wien 1090, \"Osterreich}
\email{chenxiaohang92@gmail.com, xiaohangc92@univie.ac.at}

\author[K. Dilcher]{Karl Dilcher}
\address[K. Dilcher]{Department of Mathematics and Statistics, Dalhousie University, Halifax, NS, B3H 4R2, Canada}
\email{dilcher@mathstat.dal.ca}

\author[L. Jiu]{Lin Jiu}
\address[L. Jiu]{{}\textsuperscript{(1)}Zu Chongzhi Center, Duke Kunshan University, Kunshan, Suzhou, Jiangsu Province, 215316, P.R. China \newline \indent{}\textsuperscript{(2)}Department of Mathematics and Statistics, Dalhousie University, Halifax, NS, B3H 4R2, Canada}
\email{lin.jiu@dukekunshan.edu.cn, lin.jiu.work@gmail.com}

\date{}

\keywords{Central $q$-binomial coefficient, basic hypergeometric series, $q$-Gau\ss{} summation, $q$-Chu--Vandermonde summation, $q$-Pfaff--Saalsch\"utz summation.}

\subjclass[2020]{33D15, 33D05.}

\begin{document}
	
\sloppy

\begin{abstract}
	We establish several new series evaluations involving the central $q$-binomial coefficients, with the inspiration coming from earlier work by Vignat and one of the authors on the limiting case at $q\to 1$.
\end{abstract}

\maketitle

\section{Introduction}

The \emph{central binomial coefficients} $\binom{2k}{k}$ play an important role in the theory of infinite series. A classic example arises from the Maclaurin expansion of the inverse trigonometric function $\arcsin z$, asserting that \cite[p.~121, eq.~(4.24.1)]{RO2010}:
\begin{align*}
	\arcsin z = \sum_{k\ge 0} \binom{2k}{k} \frac{z^{2k+1}}{4^k(2k+1)},
\end{align*}
for $z\in \mathbb{C}$ with $|z|\le 1$. This produces a family of series for $\pi$ including
\begin{align*}
	\frac{\pi}{2} = \sum_{k\ge 0} \binom{2k}{k} \frac{1}{4^k(2k+1)}.
\end{align*}
Along this line, Vignat and one of the authors~\cite[Theorem~2.1]{DV2025} established the following generalization for $\alpha\in \mathbb{C}\backslash \{-1,-3,-5,\ldots\}$,
\begin{align}\label{eq:DV-1}
	\sum_{k\ge 0} \binom{2k}{k}\frac{1}{4^k (2k+1+\alpha)} = \frac{\Gamma(\frac{3}{2})\;\Gamma(\frac{\alpha+1}{2})}{\Gamma(\frac{\alpha+2}{2})} = \frac{\sqrt{\pi}}{2} \frac{\Gamma(\frac{\alpha+1}{2})}{\Gamma(\frac{\alpha+2}{2})},
\end{align}
where $\Gamma(z)$ is the \emph{gamma function}.

As is often the case, it is of general interest to consider $q$-analogs of these infinite series evaluations. We recall that the binomial coefficient $\binom{n}{m}$ is the specialization at $q \to 1$ of the \emph{$q$-binomial coefficient}
\begin{align*}
	\qbinom{n}{m}_q:=\begin{cases}
		\dfrac{(q;q)_n}{(q;q)_m(q;q)_{n-m}}, & \text{if $0\le m\le n$};\\[10pt]
		0, & \text{otherwise},
	\end{cases}
\end{align*}
where the \emph{$q$-Pochhammer symbols} are defined for $N\in \mathbb{N}\cup\{\infty\}$,
\begin{align*}
	(a;q)_N&:=\prod_{j=0}^{N-1} (1-a q^j),\\
	(a_1,a_2,\ldots,a_l;q)_N &:= (a_1;q)_N (a_2;q)_N \cdots (a_l;q)_N.
\end{align*}
The main objective of this paper is to investigate series evaluations involving the central $q$-binomial coefficients.

Although shown by an elegant technique of integration in \cite{DV2025}, the series in \eqref{eq:DV-1} can be more universally understood as a ${}_2 F_1$ hypergeometric series~\cite[p.~62, eq.~(2.1.2)]{AAR1999}. Namely,
\begin{align*}
	\sum_{k\ge 0} \binom{2k}{k}\frac{1}{4^k (2k+1+\alpha)} = \frac{1}{\alpha+1}\; {}_2 F_1\left(\begin{gathered}
		\tfrac{1}{2},\tfrac{\alpha+1}{2}\\
		\tfrac{\alpha+3}{2}
	\end{gathered};1\right).
\end{align*}
This then can be evaluated to the gamma-quotient in \eqref{eq:DV-1} by the \emph{Gau\ss{} summation formula}~\cite[p.~66, Theorem~2.2.2]{AAR1999}. It is known that the ${}_2 F_1$ series can be generalized to the \emph{basic hypergeometric series}
\begin{align*}
	{}_{r}\phi_{s}\left(\begin{gathered}
		a_{1},\ldots,a_{r}\\
		b_{1},\ldots,b_{s}
	\end{gathered}
	;q,z\right) := \sum_{k\ge 0}\left((-1)^{k}q^{\binom{k}{2}}\right)^{1+s-r}\frac{(a_{1},\ldots,a_{r};q)_{k}}{(q,b_{1},\ldots,b_{s};q)_{k}}z^{k}.
\end{align*}
In view of the \emph{$q$-Gau\ss{} summation}~\cite[p.~354, eq.~(II.8)]{GR2004}:
\begin{align}\label{eq:q-Gauss}
	{}_{2}\phi_{1}\left(\begin{gathered}
		a,b\\
		c
	\end{gathered}
	;q,\frac{c}{ab}\right)=\frac{(c/a,c/b;q)_{\infty}}{(c,c/(ab);q)_{\infty}},
\end{align}
it is not surprising to see the following $q$-analog of \eqref{eq:DV-1}:
\begin{align}
	\sum_{k\ge 0} \qbinom{2k}{k}_{q^2} \frac{q^k}{(-q;q)_{2k} [2k+1+\alpha]_{q}} &= \frac{1}{[\alpha+1]_q}\; {}_{2}\phi_1 \left(\begin{gathered}
		q,q^{\alpha+1}\\
		q^{\alpha+3}
	\end{gathered}
	;q^{2},q\right)\notag\\
	&= \frac{\Gamma_{q^2}(\frac{3}{2})\; \Gamma_{q^2}(\frac{\alpha+1}{2})}{\Gamma_{q^2}(\frac{\alpha+2}{2})},
\end{align}
where we have adopted the \emph{$q$-integers}
\begin{align*}
	[n]_q := 1 + q + \cdots + q^{n-1} = \frac{1-q^n}{1-q},
\end{align*}
and the \emph{$q$-gamma function}
\begin{align*}
	\Gamma_{q}(z):=\frac{(q;q)_{\infty}}{(q^{z};q)_{\infty}}(1-q)^{1-z}.
\end{align*}

What makes our study more interesting is the singular case of the series in \eqref{eq:DV-1}. Removing the singular term, the following series evaluation connecting the \emph{alternating harmonic numbers}
\begin{align*}
	\bar{H}_{n}:= \sum_{k=1}^n \frac{(-1)^{k-1}}{k}
\end{align*}
was shown in \cite[Theorem~3.1]{DV2025}.

\begin{theorem}[Dilcher--Vignat]
	For all nonnegative integers $m$,
	\begin{align}\label{eq:DV-2}
		\sum_{\substack{k\ge 0\\k\ne m}} \binom{2k}{k}\frac{1}{4^k (k-m)} = \binom{2m}{m} \frac{\log 2 - \bar{H}_{2m}}{2^{2m-1}}.
	\end{align}
\end{theorem}

In this work, we focus on the natural $q$-analog of the series in \eqref{eq:DV-2}:
\begin{align}\label{eq:singular-q-series}
	\sum_{\substack{k\ge 0\\k\ne m}} \qbinom{2k}{k}_{q^2}\frac{q^k}{(-q;q)_{2k} [k-m]_{q^2}}.
\end{align}
Our main result is as follows.

\begin{theorem}\label{th:singular-q}
	For all nonnegative integers $m$,
	\begin{align}\label{eq:singular-q}
		\sum_{\substack{k\ge 0\\k\ne m}} \qbinom{2k}{k}_{q^2}\frac{q^k}{(-q;q)_{2k} [k-m]_{q^2}} = \qbinom{2m}{m}_{q^2}\frac{(1+q)q^m}{(-q;q)_{2m}} \sum_{k\ge 2m+1} \frac{(-1)^{k-1} q^k}{[k]_q}.
	\end{align}
\end{theorem}

\begin{remark}
	Letting $q\to 1$, we in particular have
	\begin{align*}
		\lim_{q\to 1} \sum_{k\ge 2m+1} \frac{(-1)^{k-1} q^k}{[k]_q} = \sum_{k\ge 1} \frac{(-1)^{k-1}}{k} - \sum_{k=1}^{2m} \frac{(-1)^{k-1}}{k} = \log 2 - \bar{H}_{2m},
	\end{align*}
	thereby recovering \eqref{eq:DV-2}.
\end{remark}

\textbf{Outline of the paper.} First in Section~\ref{sec:singular-q-direct}, we give a direct proof of our main result in Theorem~\ref{th:singular-q}. Meanwhile, we note that the series \eqref{eq:singular-q-series} consists of two components, one \emph{infinite} and the other \emph{finite}, and each of them should have independent interest. Thus, in Sections~\ref{sec:infinite} and \ref{sec:finite}, we work on the two summations separately --- combining the relations in Theorems~\ref{th:singular-infinite} and \ref{th:singular-finite} produces an alternative proof of Theorem~\ref{th:singular-q}, although this makes a detour in comparison with the one in Section~\ref{sec:singular-q-direct}. Finally, the evaluation for the finite component suggests some variants of Theorem~\ref{th:singular-finite}, and we will discuss them in Section~\ref{sec:var}.

\section{A direct proof of Theorem~\ref{th:singular-q}}\label{sec:singular-q-direct}

In this section, we provide a direct proof of Theorem~\ref{th:singular-q}. In doing so, we need a widely known trick of introducing auxiliary variables for series transformations. First, we note that
\begin{align}\label{eq:CentralQBinomimal2QPochhammer}
	\qbinom{2k}{k}_{q^2}\frac{1}{(-q;q)_{2k}} = \frac{(q;q^2)_k}{(q^2;q^2)_k}.
\end{align}
Therefore,
\begin{align*}
	\sum_{\substack{k\ge 0\\k\ne m}} \qbinom{2k}{k}_{q^2}\frac{q^k}{(-q;q)_{2k} [k-m]_{q^2}} = (1-q^2) \sum_{\substack{k\ge 0\\k\ne m}} \frac{(q;q^2)_{k}q^{k}}{(q^2;q^2)_{k}(1-q^{2k-2m})}.
\end{align*}
Now let us define
\begin{align*}
	F(a,q) := \frac{1}{1-aq^{-2m}}\left({}_{2}\phi_1 \left(\begin{gathered}
		q,aq^{-2m}\\
		aq^{-2m+2}
	\end{gathered}
	;q^{2},q\right) - \frac{(q,aq^{-2m};q^2)_m q^m}{(q^2,aq^{-2m+2};q^2)_m}\right).
\end{align*}
In other words,
\begin{align*}
	F(a,q) = \sum_{\substack{k\ge 0\\k\ne m}} \frac{(q;q^2)_{k}q^{k}}{(q^2;q^2)_{k}(1-aq^{2k-2m})},
\end{align*}
so that
\begin{align*}
	\sum_{\substack{k\ge 0\\k\ne m}} \qbinom{2k}{k}_{q^2}\frac{q^k}{(-q;q)_{2k} [k-m]_{q^2}} = (1-q^2) \lim_{a\to 1} F(a,q).
\end{align*}
For the ${}_2 \phi_1$ series in $F(a,q)$, the $q$-Gau\ss{} summation \eqref{eq:q-Gauss} gives us
\begin{align*}
	{}_{2}\phi_1 \left(\begin{gathered}
		q,aq^{-2m}\\
		aq^{-2m+2}
	\end{gathered}
	;q^{2},q\right) = \frac{(q^2,aq^{-2m+1};q^2)_\infty}{(q,aq^{-2m+2};q^2)_\infty}.
\end{align*}
Thus,
\begin{align*}
	F(a,q) &= \frac{1}{1-aq^{-2m}}\left(\frac{(q^2,aq^{-2m+1};q^2)_\infty}{(q,aq^{-2m+2};q^2)_\infty} - \frac{(q,aq^{-2m};q^2)_m q^m}{(q^2,aq^{-2m+2};q^2)_m}\right)\\
	&= \frac{1}{1-a} \left(\frac{(q^2,aq;q^2)_\infty (aq^{-2m+1};q^2)_m}{(q,aq^2;q^2)_\infty (aq^{-2m};q^2)_m} - \frac{(q;q^2)_m q^m}{(q^2;q^2)_m}\right).
\end{align*}
Since
\begin{align*}
	\lim_{a\to 1} \frac{(q^2,aq;q^2)_\infty (aq^{-2m+1};q^2)_m}{(q,aq^2;q^2)_\infty (aq^{-2m};q^2)_m} = \frac{(q^{-2m+1};q^2)_m}{(q^{-2m};q^2)_m} = \frac{(q;q^2)_m q^m}{(q^2;q^2)_m},
\end{align*}
we may apply the L'Hospital rule to get
\begin{align*}
	\lim_{a\to 1} F(a,q) &= - \frac{\partial}{\partial a}\bigg\vert_{a=1} \left(\frac{(q^2,aq;q^2)_\infty (aq^{-2m+1};q^2)_m}{(q,aq^2;q^2)_\infty (aq^{-2m};q^2)_m} - \frac{(q;q^2)_m q^m}{(q^2;q^2)_m}\right)\\
	&= - \frac{(q^2;q^2)_\infty}{(q;q^2)_\infty} \frac{\partial}{\partial a}\bigg\vert_{a=1} \frac{(aq;q^2)_\infty (aq^{-2m+1};q^2)_m}{(aq^2;q^2)_\infty (aq^{-2m};q^2)_m}.
\end{align*}
Recall that for a generic function $f(a)$,
\begin{align*}
	\frac{\dd}{\dd a} f(a) = f(a) \frac{\dd}{\dd a} \log f(a).
\end{align*}
We then compute that
\begin{align*}
	\lim_{a\to 1} F(a,q) &= - \frac{(q^{-2m+1};q^2)_m}{(q^{-2m};q^2)_m} \frac{\partial}{\partial a}\bigg\vert_{a=1} \left(\sum_{k\ge 1}\log \frac{1-aq^{2k-1}}{1-aq^{2k}} + \sum_{k=1}^m\log \frac{1-aq^{-2k+1}}{1-aq^{-2k}}\right)\\
	&= \frac{(q;q^2)_m q^m}{(q^2;q^2)_m} \left(\sum_{k\ge 1} \frac{(-1)^{k-1}q^k}{1-q^{k}} + \sum_{k=1}^{2m} \frac{(-1)^{k-1}q^{-k}}{1-q^{-k}}\right)\\
	&= \frac{(q;q^2)_m q^m}{(q^2;q^2)_m} \left(\sum_{k\ge 2m+1} \frac{(-1)^{k-1}q^k}{1-q^{k}} + \sum_{k=1}^{2m} (-1)^k\right)\\
	&= \frac{(q;q^2)_m q^m}{(q^2;q^2)_m} \sum_{k\ge 2m+1} \frac{(-1)^{k-1}q^k}{1-q^{k}}.
\end{align*}
Finally,
\begin{align*}
	\sum_{\substack{k\ge 0\\k\ne m}} \qbinom{2k}{k}_{q^2}\frac{q^k}{(-q;q)_{2k} [k-m]_{q^2}} &= (1-q^2) \lim_{a\to 1} F(a,q)\\
	& = \frac{(q;q^2)_m (1+q) q^m}{(q^2;q^2)_m} \sum_{k\ge 2m+1} \frac{(-1)^{k-1}q^k}{[k]_q},
\end{align*}
which is exactly the right-hand side of \eqref{eq:singular-q} after rewriting the prefactor in terms of the central $q$-binomial coefficient according to \eqref{eq:CentralQBinomimal2QPochhammer}.

\section{The infinite component}\label{sec:infinite}

Now we move on to the infinite component in \eqref{eq:singular-q-series}, which also exhibits a neat evaluation.

\begin{theorem}\label{th:singular-infinite}
	For all nonnegative integers $m$,
	\begin{align}\label{eq:singular-infinite}
		&\sum_{k\ge m+1} \qbinom{2k}{k}_{q^2}\frac{q^k}{(-q;q)_{2k} [k-m]_{q^2}}\notag\\
		&\qquad = \qbinom{2m}{m}_{q^2}\frac{(1+q)q^m}{(-q;q)_{2m}} \left(\sum_{k\ge 1} \frac{q^{2k-1}}{[2k-1]_q} - \sum_{k\ge 1} \frac{q^{2k+2m}}{[2k+2m]_q}\right).
	\end{align}
\end{theorem}

\begin{proof}
	We have, by again applying \eqref{eq:CentralQBinomimal2QPochhammer} and then shifting the index $k\mapsto k+m$,
	\begin{align*}
		\sum_{k\ge m+1} \qbinom{2k}{k}_{q^2}\frac{q^k}{(-q;q)_{2k} [k-m]_{q^2}} &= (1-q^2) \sum_{k\ge 1} \frac{(q;q^2)_{k+m}q^{k+m}}{(q^2;q^2)_{k+m}(1-q^{2k})}\\
		&= \frac{(1-q^2)q^m(q;q^2)_m}{(q^2;q^2)_m} \sum_{k\ge 1} \frac{(q^{2m+1};q^2)_k q^k}{(q^{2m+2};q^2)_k(1-q^{2k})}\\
		&= \qbinom{2m}{m}_{q^2}\frac{(1-q^2)q^m}{(-q;q)_{2m}} \sum_{k\ge 1} \frac{(q^{2m+1};q^2)_k q^k}{(q^{2m+2};q^2)_k(1-q^{2k})}.
	\end{align*}
	For the moment, we introduce the auxiliary series
	\begin{align*}
		U(a,q) := \frac{1}{1-a} \left({}_{2}\phi_1 \left(\begin{gathered}
			a,q^{2m+1}\\
			q^{2m+2}
		\end{gathered}
		;q^{2},a^{-1}q\right) - 1\right).
	\end{align*}
	It is clear that
	\begin{align*}
		\lim_{a\to 1} U(a,q) &= \lim_{a\to 1} \frac{1}{1-a} \sum_{k\ge 1} \frac{(a,q^{2m+1};q^2)_k (a^{-1}q)^k}{(q^2,q^{2m+2};q^2)_k}\\
		&= \left[\sum_{k\ge 1} \frac{(aq^2;q^2)_{k-1}(q^{2m+1};q^2)_k (a^{-1}q)^k}{(q^2,q^{2m+2};q^2)_k}\right]_{a=1}\\
		&= \sum_{k\ge 1} \frac{(q^{2m+1};q^2)_k q^k}{(q^{2m+2};q^2)_k(1-q^{2k})}.
	\end{align*}
	Now applying the $q$-Gau\ss{} summation \eqref{eq:q-Gauss} to the ${}_2 \phi_1$ series in $U(a,q)$, we obtain
	\begin{align*}
		U(a,q) = \frac{1}{1-a} \left(\frac{(q,a^{-1}q^{2m+2};q^2)_\infty}{(a^{-1}q,q^{2m+2};q^2)_\infty}-1\right).
	\end{align*}
	By L'Hospital,
	\begin{align*}
		\lim_{a\to 1} U(a,q) &= - \frac{\partial}{\partial a}\bigg\vert_{a=1} \frac{(q,a^{-1}q^{2m+2};q^2)_\infty}{(a^{-1}q,q^{2m+2};q^2)_\infty}\\
		&= - \left[\frac{(q,a^{-1}q^{2m+2};q^2)_\infty}{(a^{-1}q,q^{2m+2};q^2)_\infty}\sum_{k\ge 1} \frac{\partial}{\partial a} \log \frac{1-a^{-1}q^{2k+2m}}{1-a^{-1}q^{2k-1}}\right]_{a=1}\\
		&= \sum_{k\ge 1} \frac{q^{2k-1}}{1-q^{2k-1}} - \sum_{k\ge 1} \frac{q^{2k+2m}}{1-q^{2k+2m}}.
	\end{align*}
	The desired relation then follows.
\end{proof}

\section{The finite component}\label{sec:finite}

Although the finite component in \eqref{eq:singular-q-series} can be computed by subtracting \eqref{eq:singular-infinite} from \eqref{eq:singular-q}, we offer an independent proof for its own interest.

\begin{theorem}\label{th:singular-finite}
	For all nonnegative integers $m$,
	\begin{align}\label{eq:singular-finite}
		\sum_{k=0}^{m-1} \qbinom{2k}{k}_{q^2}\frac{q^k}{(-q;q)_{2k} [k-m]_{q^2}} = -\qbinom{2m}{m}_{q^2}\frac{(1+q)q^m}{(-q;q)_{2m}} \sum_{k=1}^m \frac{q^{2k-1}}{[2k-1]_q}.
	\end{align}
\end{theorem}

\begin{proof}
	Note that by \eqref{eq:CentralQBinomimal2QPochhammer},
	\begin{align*}
		\sum_{k=0}^{m-1} \qbinom{2k}{k}_{q^2}\frac{q^k}{(-q;q)_{2k} [k-m]_{q^2}} = (1-q^2) \sum_{k=0}^{m-1} \frac{(q;q^2)_{k}q^{k}}{(q^2;q^2)_{k}(1-q^{2k-2m})}.
	\end{align*}
	Let us define the auxiliary series
	\begin{align*}
		V(a,q) := \frac{1}{1-q^{-2m}} \left({}_{2}\phi_1 \left(\begin{gathered}
			q,q^{-2m}\\
			a^{-1}q^{-2m+2}
		\end{gathered}
		;q^{2},a^{-1}q\right) - \frac{(q,q^{-2m};q^2)_m (a^{-1}q)^m}{(q^2,a^{-1}q^{-2m+2};q^2)_m}\right).
	\end{align*}
	In particular,
	\begin{align*}
		V(a,q) = \sum_{k=0}^{m-1} \frac{(q,q^{-2m};q^2)_k (a^{-1}q)^k}{(q^2,a^{-1}q^{-2m+2};q^2)_k},
	\end{align*}
	so that
	\begin{align*}
		\lim_{a\to 1} V(a,q) = \sum_{k=0}^{m-1} \frac{(q;q^2)_{k}q^{k}}{(q^2;q^2)_{k}(1-q^{2k-2m})}.
	\end{align*}
	Now for the terminating ${}_2 \phi_1$ series in $V(a,q)$, we apply the \emph{first $q$-Chu--Vandermonde summation}~\cite[p.~354, eq.~(II.7)]{GR2004}:
	\begin{align*}
		{}_{2} \phi_{1} \left(\begin{matrix}
			a,q^{-N}\\
			c
		\end{matrix};q,\frac{cq^N}{a}\right) = \frac{(c/a;q)_N}{(c;q)_N}.
	\end{align*}
	Thus,
	\begin{align*}
		V(a,q) &= \frac{1}{1-q^{-2m}} \left(\frac{(a^{-1}q^{-2m+1};q^2)_m}{(a^{-1}q^{-2m+2};q^2)_m} - \frac{(q,q^{-2m};q^2)_m (a^{-1}q)^m}{(q^2,a^{-1}q^{-2m+2};q^2)_m}\right)\\
		&= \frac{1}{1-q^{-2m}} \left(\frac{(aq;q^2)_m q^{-m}}{(a;q^2)_m} - \frac{(q;q^2)_m q^{-m}}{(a;q^2)_m}\right)\\
		&= -\frac{q^m}{1-q^{2m}} \frac{(aq;q^2)_m - (q;q^2)_m}{(1-a)(aq^2;q^2)_{m-1}}.
	\end{align*}
	It follows that
	\begin{align*}
		\lim_{a\to 1} V(a,q) &= -\frac{q^m}{1-q^{2m}} \left(\lim_{a\to 1} \frac{1}{(aq^2;q^2)_{m-1}} \right)\left(\lim_{a\to 1} \frac{(aq;q^2)_m - (q;q^2)_m}{1-a}\right)\\
		&= -\frac{q^m}{(q^2;q^2)_m} \lim_{a\to 1} \frac{(aq;q^2)_m - (q;q^2)_m}{1-a}.
	\end{align*}
	By L'Hospital,
	\begin{align*}
		\lim_{a\to 1} V(a,q) &= \frac{q^m}{(q^2;q^2)_m} \frac{\partial}{\partial a}\bigg\vert_{a=1} (aq;q^2)_m\\
		&= \frac{q^m}{(q^2;q^2)_m} \left[(aq;q^2)_m \sum_{k=1}^m \frac{\partial}{\partial a} \log(1-aq^{2k-1})\right]_{a=1}\\
		&= -\frac{(q;q^2)_m q^m}{(q^2;q^2)_m} \sum_{k=1}^m \frac{q^{2k-1}}{1-q^{2k-1}}.
	\end{align*}
	Therefore,
	\begin{align*}
		\sum_{k=0}^{m-1} \qbinom{2k}{k}_{q^2}\frac{q^k}{(-q;q)_{2k} [k-m]_{q^2}} &= (1-q^2) \lim_{a\to 1} V(a,q)\\
		&= -\qbinom{2m}{m}_{q^2}\frac{(1-q^2)q^m}{(-q;q)_{2m}} \sum_{k=1}^m \frac{q^{2k-1}}{1-q^{2k-1}},
	\end{align*}
	as claimed.
\end{proof}

\section{Finite variants}\label{sec:var}

In the proof of Theorem~\ref{th:singular-finite}, we have applied the first $q$-Chu--Vandermonde summation. Notably, this formula has a natural companion, which suggests the following variant of \eqref{eq:singular-finite}.

\begin{theorem}\label{th:singular-finite-variant}
	For all nonnegative integers $m$,
	\begin{align}\label{eq:singular-finite-variant}
		\sum_{k=0}^{m-1} \qbinom{2k}{k}_{q^2}\frac{q^{2k}}{(-q;q)_{2k} [k-m]_{q^2}} = -\qbinom{2m}{m}_{q^2}\frac{(1+q)q^{2m}}{(-q;q)_{2m}} \sum_{k=1}^m \frac{1}{[2k-1]_q}.
	\end{align}
\end{theorem}

\begin{proof}
	In view of \eqref{eq:CentralQBinomimal2QPochhammer}, this time we have
	\begin{align*}
		\sum_{k=0}^{m-1} \qbinom{2k}{k}_{q^2}\frac{q^{2k}}{(-q;q)_{2k} [k-m]_{q^2}} = (1-q^2) \sum_{k=0}^{m-1} \frac{(q;q^2)_{k}q^{2k}}{(q^2;q^2)_{k}(1-q^{2k-2m})}.
	\end{align*}
	Define the auxiliary series
	\begin{align*}
		\widehat{V}(a,q) := \frac{1}{1-q^{-2m}} \left({}_{2}\phi_1 \left(\begin{gathered}
			q,q^{-2m}\\
			a^{-1}q^{-2m+2}
		\end{gathered}
		;q^{2},q^2\right) - \frac{(q,q^{-2m};q^2)_m q^{2m}}{(q^2,a^{-1}q^{-2m+2};q^2)_m}\right),
	\end{align*}
	which satisfies
	\begin{align*}
		\lim_{a\to 1} \widehat{V}(a,q) = \sum_{k=0}^{m-1} \frac{(q;q^2)_{k}q^{2k}}{(q^2;q^2)_{k}(1-q^{2k-2m})}.
	\end{align*}
	For the terminating ${}_2 \phi_1$ series in $\widehat{V}(a,q)$, we need the \emph{second $q$-Chu--Vandermonde summation}~\cite[p.~354, eq.~(II.6)]{GR2004}:
	\begin{align*}
		{}_{2} \phi_{1} \left(\begin{matrix}
			a,q^{-N}\\
			c
		\end{matrix};q,q\right) = \frac{(c/a;q)_N a^N}{(c;q)_N}.
	\end{align*}
	Thus,
	\begin{align*}
		\widehat{V}(a,q) &= \frac{1}{1-q^{-2m}} \left(\frac{(a^{-1}q^{-2m+1};q^2)_m q^m}{(a^{-1}q^{-2m+2};q^2)_m} - \frac{(q,q^{-2m};q^2)_m q^{2m}}{(q^2,a^{-1}q^{-2m+2};q^2)_m}\right)\\
		&= -\frac{q^{2m}}{1-q^{2m}} \frac{(aq;q^2)_m - (q;q^2)_m a^m}{(1-a)(aq^2;q^2)_{m-1}},
	\end{align*}
	which further gives
	\begin{align*}
		\lim_{a\to 1} \widehat{V}(a,q) &= -\frac{q^{2m}}{(q^2;q^2)_m} \lim_{a\to 1} \frac{(aq;q^2)_m - (q;q^2)_m a^m}{1-a}\\
		&= \frac{(q;q^2)_mq^{2m}}{(q^2;q^2)_m} \left(-m + \sum_{k=1}^m \frac{\partial}{\partial a}\bigg\vert_{a=1} \log(1-aq^{2k-1})\right)\\
		&= -\frac{(q;q^2)_m q^{2m}}{(q^2;q^2)_m} \sum_{k=1}^m \frac{1}{1-q^{2k-1}},
	\end{align*}
	where we have again applied the L'Hospital rule. The required identity then follows by substituting this limit into the first equation in the proof.
\end{proof}

In light of Theorems~\ref{th:singular-finite} and \ref{th:singular-finite-variant}, we derive an intriguing identity.

\begin{corollary}
	For all nonnegative integers $m$,
	\begin{align}\label{eq:singular-finite-coro}
		\sum_{k=0}^{m-1} \qbinom{2k}{k}_{q^2}\frac{q^{k}}{(-q;q)_{2k} (1+q^{k-m})} = \qbinom{2m}{m}_{q^2}\frac{mq^{2m}}{(-q;q)_{2m}}.
	\end{align}
\end{corollary}

\begin{proof}
	The claimed result follows by dividing by $q^m$ on both sides of \eqref{eq:singular-finite-variant} and then taking its difference with \eqref{eq:singular-finite}.
\end{proof}

\begin{remark}
	The limiting case of \eqref{eq:singular-finite-coro} at $q\to 1$ is
	\begin{align}
		\sum_{k=0}^{m-1} \binom{2k}{k} \frac{1}{2^{2k+1}} = \binom{2m}{m} \frac{m}{4^m}.
	\end{align}
	This identity can be found in \cite[p.~14, eq.~(1.109)]{Gou1972}.
\end{remark}

Besides \eqref{eq:singular-finite-variant}, the $q$-Pfaff--Saalsch\"utz summation produces a second variant of \eqref{eq:singular-finite} in a \emph{squared} manner.

\begin{theorem}\label{th:singular-finite-variant-2}
	For all nonnegative integers $m$,
	\begin{align}\label{eq:singular-finite-variant-2}
		\sum_{k=0}^{m-1} \qbinom{2k}{k}_{q^2}^2\frac{q^{2k}}{(-q;q)_{2k}^2 [k-m]_{q^2}} = -\qbinom{2m}{m}_{q^2}^2\frac{(1+q)q^{2m}}{(-q;q)_{2m}^2} \sum_{k=1}^m \frac{[4k-2]_q}{[2k-1]_q^2}.
	\end{align}
\end{theorem}

\begin{proof}
	Note that by \eqref{eq:CentralQBinomimal2QPochhammer},
	\begin{align*}
		\sum_{k=0}^{m-1} \qbinom{2k}{k}_{q^2}^2\frac{q^{2k}}{(-q;q)_{2k}^2 [k-m]_{q^2}} = (1-q^2) \sum_{k=0}^{m-1} \frac{(q,q;q^2)_{k}q^{2k}}{(q^2,q^2;q^2)_{k}(1-q^{2k-2m})}.
	\end{align*}
	Now we introduce the auxiliary series
	\begin{align*}
		W(a,q) := \frac{1}{1-q^{-2m}} \left({}_{3}\phi_2 \left(\begin{gathered}
			q,a^{-1}q,q^{-2m}\\
			q^2,a^{-1}q^{-2m+2}
		\end{gathered}
		;q^{2},q^2\right) - \frac{(q,a^{-1}q,q^{-2m};q^2)_m q^{2m}}{(q^2,q^2,a^{-1}q^{-2m+2};q^2)_m}\right),
	\end{align*}
	so that
	\begin{align*}
		\lim_{a\to 1} W(a,q) = \sum_{k=0}^{m-1} \frac{(q,q;q^2)_{k}q^{2k}}{(q^2,q^2;q^2)_{k}(1-q^{2k-2m})}.
	\end{align*}
	For the terminating ${}_3 \phi_2$ series in $W(a,q)$, we require the \emph{$q$-Pfaff--Saalsch\"utz summation}~\cite[p.~355, eq.~(II.12)]{GR2004}:
	\begin{align*}
		{}_{3} \phi_{2} \left(\begin{matrix}
			a,b,q^{-N}\\
			c,abc^{-1}q^{1-N}
		\end{matrix};q,q\right) = \frac{(c/a,c/b;q)_N}{(c,c/(ab);q)_N}.
	\end{align*}
	Then we arrive at
	\begin{align*}
		W(a,q) &= \frac{1}{1-q^{-2m}} \left(\frac{(q,aq;q^2)_m}{(q^2,a;q^2)_m} - \frac{(q,a^{-1}q,q^{-2m};q^2)_m q^{2m}}{(q^2,q^2,a^{-1}q^{-2m+2};q^2)_m}\right)\\
		&= -\frac{q^{2m}}{1-q^{2m}} \frac{(q;q^2)_m}{(q^2;q^2)_m} \frac{(aq;q^2)_m - (a^{-1}q;q^2)_m a^m}{(1-a)(aq^2;q^2)_{m-1}},
	\end{align*}
	thereby yielding
	\begin{align*}
		\lim_{a\to 1} W(a,q) &= -\frac{(q;q^2)_m q^{2m}}{(q^2;q^2)_m^2} \lim_{a\to 1} \frac{(aq;q^2)_m - (a^{-1}q;q^2)_m a^m}{1-a}\\
		&= \frac{(q;q^2)_m^2 q^{2m}}{(q^2;q^2)_m^2} \sum_{k=1}^m \frac{\partial}{\partial a}\bigg\vert_{a=1} \log(1-aq^{2k-1})\\
		&\quad - \frac{(q;q^2)_m^2 q^{2m}}{(q^2;q^2)_m^2} \left(m+\sum_{k=1}^m \frac{\partial}{\partial a}\bigg\vert_{a=1} \log(1-a^{-1}q^{2k-1})\right)\\
		&= -\frac{(q;q^2)_m^2 q^{2m}}{(q^2;q^2)_m^2} \sum_{k=1}^m \frac{1+q^{2k-1}}{1-q^{2k-1}}.
	\end{align*}
	Finally, recalling that
	\begin{align*}
		\sum_{k=0}^{m-1} \qbinom{2k}{k}_{q^2}^2\frac{q^{2k}}{(-q;q)_{2k}^2 [k-m]_{q^2}} = (1-q^2) \lim_{a\to 1} W(a,q),
	\end{align*}
	the claimed result holds after a modest amount of elementary algebra.
\end{proof}

\subsection*{Acknowledgements}

Shane Chern was supported by the FWF Austrian Science Fund (No.~10.55776/F1002). Lin Jiu would like to thank Dr.~Chen Zhang at Duke Kunshan University for personal financial support of his living expenses on visiting Dalhousie University, during which period the present work was completed. 

\bibliographystyle{amsplain}

\begin{thebibliography}{9}
	
	\bibitem{AAR1999}
	G. E. Andrews, R. Askey, and R. Roy, \textit{Special functions}, Cambridge University Press, Cambridge, 1999.
	
	\bibitem{DV2025}
	K. Dilcher and C. Vignat, Further classes of series involving central binomial coefficients, preprint. \arxiv{2511.00109}.
	
	\bibitem{GR2004}
	G. Gasper and M. Rahman, \textit{Basic hypergeometric series. Second edition}, Cambridge University Press, Cambridge, 2004.
	
	\bibitem{Gou1972}
	H. W. Gould, \textit{Combinatorial identities}, Morgantown Printing and Binding Co., Morgantown, WV, 1972.
	
	\bibitem{RO2010}
	R. Roy and F. W. J. Olver, Elementary functions, in: \textit{NIST Handbook of Mathematical Functions}, 103--134, Cambridge University Press, Cambridge, 2010.
	
\end{thebibliography}

\end{document}